\documentclass[pamsfonts, 11pt,a4paper]{amsart}

\usepackage{amsfonts, amsmath, amsthm, amssymb}
\usepackage[mathscr]{eucal}
\usepackage{bbm}
\usepackage{xcolor}
\usepackage{url}

\newcommand{\F}{\mathbb{F}}

\newtheorem{theorem}{Theorem}

\theoremstyle{definition}

\newtheorem*{example}{Example}
\theoremstyle{remark}

\begin{document}
\title{A note on powers in finite fields}
\author{Andreas Aabrandt}
\author{Vagn Lundsgaard Hansen}
\address{Technical University of Denmark}
\maketitle

\vspace{-30pt}

\begin{abstract}
The study of solutions to polynomial equations over finite fields has a long history in mathematics and is an interesting area of contemporary research. In recent years the subject has found important applications in the modelling of problems from applied mathematical fields such as signal analysis, system theory, coding theory and cryptology. In this connection it is of interest to know criteria for the existence of squares and other powers in arbitrary finite fields. Making good use of polynomial division in polynomial rings over finite fields, we have examined a classical criterion of Euler for squares in odd prime fields, giving it a formulation which is apt for generalization to arbitrary finite fields and powers. Our proof uses algebra rather than classical number theory, which makes it convenient when presenting basic methods of applied algebra in the classroom. 
\end{abstract}

\subjclass{{\small \textbf{Subject class:} 11A15, 12C15}}

\keywords{{\small \textbf{Keywords:} Finite fields,  prime numbers, squares and powers in finite fields}}

\section*{}

\vspace{-25pt}

For any prime number $p$ we denote by $\F_p$ the field of residue classes of integers modulo $p$. And more generally, for any integer $n\ge 1$, we denote by $\F_q$ the finite field with $q=p^n$ elements. We identify $\F_p$ with the prime field of $\F_q$ by identifying $1\in\F_p$ with the identity element $e\in\F_q$. Denote by $\F_q^*$ the group of nonzero elements in $\F_q$.

For the prime fields $\F_p$, the following result was known to Euler; \cite{davenport}, \cite{IrelandRosen}.

\begin{theorem}[Euler's Criterion]
Let $p$ be an odd prime, and let $a\in \F_p^*$. Then there exists an element $x\in \F_p^*$ such that $a=x^2$ in $\F_p^*$ if and only if
$$
a^{(p-1)/2}\equiv 1 \pmod p.
$$
\end{theorem}

This criterion can be generalized to all finite fields as follows.

\begin{theorem}[Generalized Euler's Criterion]
Let $q=p^n$ for an odd prime $p$ and an arbitrary integer $n\ge 1$, and let $a\in\F_q^*$. Then there exists an element $x\in\F_q^*$ such that $a=x^2$ in $\F_q^*$  if and only if
$$
a^{(q-1)/2}=1.
$$
\end{theorem}

\begin{proof}
The finite field $\F_q$ is uniquely determined up to an isomorphism as the splitting field for the polynomial $f(x)=x^q-x$  over $\F_p$, and consists of the $q$ roots of $f(x)$; see \cite{lang}.

From this description follows easily that there exists an element $x\in\F_q^*$ such that $a=x^2 \in \F_q^*$ if and only if the polynomial $g(x) = x^2 - a$ splits into linear factors and hence is a divisor in $f(x)=x^q-x$.

By polynomial division we get
$$
f(x)=x^q-x = h(x)(x^2-a)+ (a^{(q-1)/2}-1)x,
$$
where
$$
h(x) = x^{q-2}+ax^{q-4}+a^2x^{q-6}+\cdots +a^{(q-3)/2}x.
$$

From the above follows immediately that $g(x)$ is a divisor in $f(x)$ and hence that there exists an element $x\in\F_q^*$ such that $x^2= a$ if and only if 
$$
a^{(q-1)/2}=1 .
$$
\end{proof}

\begin{example}
For $p=3$, the series of fields $\F_{3^n}$ for $n\ge 1$, exhibits interesting phenomena for the existence of nontrivial squares, i.e. squares $\neq 0, 1$.

The finite field $\F_3$ contains no nontrivial squares as can easily be seen by direct computations, or, by using Euler's criterion.

The finite field $\F_{3^2}$ can be described as the polynomial ring $\F_3[t]$ modulo the irreducible polynomial $t^2+1$. The nine elements in $\F_{3^2}$ are then uniquely described by the nine polynomials $x=a_0+a_1t$, for $a_0,a_1\in \F_3$, and $t^2+1 = 0$. It follows that $t^2=-1=2$ and $(1+2t)^2=1+4t+4t^2=t$, showing that 2, t, and 2t  are the three nontrivial  squares in  $\F_{3^2}$. This is in accordance with the generalized Euler's criterion. 

Using the generalized Euler's criterion, it can be proved by induction that 2 is a square in the finite field $\F_{3^n}$ for an arbitrary integer $n\ge 1$ if and only if $n$ is even. The formula $2^{(3^n-1)/2}=(2^{(3^{n-2}-1)/2})^9$, which holds in $\F_3$, is useful for the induction step.

The fields $\F_2$ and $\F_3$ are the only prime fields without nontrivial squares.
\end{example}

The number of nontrivial squares in an arbitrary finite field is given by

\begin{theorem}\label{thm_num_sq}
 The finite field $\F_q$,  with $q=p^n$ for an odd prime $p$ and an integer $n\ge 1$, contains $$\frac{q-1}{2}-1$$ nontrivial squares.
\end{theorem}

\begin{proof}
The multiplicative group $\F_{q}^*$ is a cyclic group of order $q-1$ generated by an element $\gamma\in\F_{q}^*$; see \cite{lang}. Every element in $\F_{q}^*$ has then a unique presentation as a power ${\gamma}^k$ of $\gamma$, where the exponent $k$ is counted modulo $q$.

The squaring homomorphism $x^2: \F_q^*\to \F_q^*$ maps the element $a={\gamma}^l\in\F_q^*$ into $c={\gamma}^{2l}\in\F_q^*$. From this we conclude that $c={\gamma}^k$ is a square in $\F_q^*$ if and only if $k$ is even modulo $q$. Hence there are equally many squares and non-squares in $\F_q^*$. Now the theorem follows immediately.
\end{proof}

It is not easy to find a square root of an element in the finite field $\F_{p^n}$ for large primes $p$ and integers $n\ge 1$.  Even using the generalized Euler's criterion it is fairly complicated just to establish that an element has a square root without the use of a computer.

\begin{example}
Consider the finite field $\F_{13^3}$.  Using the irreducible polynomial $t^3 + 2t + 11$ in the polynomial ring $\F_{13}[t]$, the field can be described as

$$
\F_{13^3} = \F_{13}[t]/(t^3 + 2t + 11).
$$

By the generalized Euler's criterion we have
$$
(5+t+8t^2)^{(13^3-1)/2} = (5+t+8t^2)^{1098} = 1,
$$
showing that  $5+t+8t^2$ is a square. And indeed, 
$$
(7+t+2t^2)^2=5+t+8t^2.
$$
\end{example}

\vspace{5pt}
The results on squares in a finite field $\F_q$ can be extended to higher powers. This leads to 
our main result in Theorem 4. It should be mentioned that Theorem 4 can be derived from the similar Proposition 7.1.2 in  \cite{IrelandRosen}. Our proof follows, however, different lines 
and uses algebra rather than classical number theory, which in our opinion makes the proof more transparent.

Let $r$ be a positive integer in the interval $2\le r<q-1$. If $r$ is not a divisor in the order $q-1$  of the multiplicative group $\F_{q}^*$, then the $r^{\rm th}$ power homomorphism $x^r: \F_q^*\to \F_q^*$ is injective and hence an isomorphism in  $\F_{q}^*$. If $r$ is not a divisor in $q-1$, every element in $\F_{q}^*$ is therefore an  $r^{\rm th}$ power.

In case $r$ is a divisor in $q-1$ we have

\begin{theorem}[Extended Euler's Criterion]
Let $q=p^n$ for an odd prime $p$ and an arbitrary integer $n\ge 1$. Suppose the number $r$ is a proper divisor in $q-1$. Then it holds that
\begin{enumerate}
\item  An element $a\in\F_q^*$ is an $r^{\rm th}$ power in $\F_q^*$ if and only if $a^{(q-1)/r}=1.$
\item There are exactly $\frac{q-1}{r}$  different $r^{\rm th}$ powers in $\F_q^*.$
\end{enumerate}
\end{theorem}

\begin{proof}
We exploit again that $\F_q$ is the splitting field for the polynomial $f(x)=x^q-x$ in $\F_p$, and is generated by the $q$ roots of $f(x)$; see \cite{lang}. 

The multiplicative group $\F_q^*$ is a cyclic group of order $q-1$. Choose a generator $\gamma \in \F_q^*$. Then the powers ${\gamma}^i$, $i=1,\dots, q-1$, runs through all the elements in $\F_q^*$. 

Let  $r$ be a proper divisor in $q-1$ and put $k=\frac {q-1}{r}$.

The $r^{\rm th}$ power homomorphism $x^r: \F_q^*\to \F_q^*$ maps ${\gamma}^i$ into ${\gamma}^{ir}$, showing that the $k$ elements $a_j={\gamma}^{jr}$, $j=1,\dots, k$, are exactly the $r^{\rm th}$ powers in $\F_q^*$. This proves part (2) of the theorem.

For $a\in\F_q^*$, consider the polynomial $g(x) = x^r - a$ in $\F_q$. 

By polynomial division in $\F_q[x]$ we get
$$
f(x)=x^q-x = h(x)(x^r-a)+ (a^{(q-1)/r}-1)x,
$$
where
$$
h(x) = x^{q-r}+ax^{q-2r}+a^2x^{q-3r}+\cdots +a^{(q-3)/r}x.
$$
It follows that $g(x)$ is a divisor in $f(x)$ if and only if $a^{(q-1)/r}=1.$

Since 
$$
f(x)=x^q-x=x(x- {\gamma}^1)(x- {\gamma}^2)\dots ( x-{\gamma}^{q-1}),
$$
it is clear that $g(x)= x^r - a$ is a divisor in $f(x)$ if and only if it splits completely into linear factors, or in other words, if and only if $a\in\F_q^*$ is the  $r^{\rm th}$ power of $r$ roots ${\gamma}^{i_1},\dots, {\gamma}^{i_r}\in\F_q^*$ of $f(x)$.

Combining information we conclude that $a\in \F_q^*$ is an $r^{\rm th}$ power in $\F_q^*$ if and only if $a^{(q-1)/r}=1$, thus proving part (1) of the theorem.
\end{proof}

\vspace{5pt}

\begin{example}
The finite field $\F_{5^2}$ can be described as the polynomial ring $\F_5[t]$ modulo the irreducible polynomial $t^2+t+1$, i.e.  $\F_{5^2} = \F_{5}[t]/(t^2 + t + 1).$

For each of the proper divisors $r = 2, 3, 4, 6, 8, 12$ in 24, which is the order of the multiplicative group $\F_{5^2}^*$,  we can using the extended Euler's criterion by hand determine which elements in the prime field $\F_5$ are $r^{\rm th}$ powers in $\F_{5^2}^*$.\\[-10pt]

\underline{$r=2, 3, 6$} 

All elements in  $\F_{5}^*$ are $r^{\rm th}$ powers in $\F_{5^2}^*$. 

In fact: \ $2=(1+2t)^2=3^3=(3+t)^6$. 

\qquad\qquad $3=(3+t)^2=2^3=(1+2t)^6$. 

\qquad\qquad $4=2^2=4^3=2^6$.\\[-10pt]

\underline{$r=4$} 

The elements $1,4$ are 4-powers in $\F_{5^2}^*$,  and $2,3$ are not 4-powers in $\F_{5^2}^*$.

In fact: \ $4=(1+2t)^4$.\\[-10pt]

\underline{$r=8$} 

The elements $2,3,4$ are not 8-powers in $\F_{5^2}^*$.\\[-10pt]

\underline{$r=12$} 

The elements $1, 4$ are 12-powers in $\F_{5^2}^*$, and $2, 3$ are not 12-powers in $\F_{5^2}^*$.

In fact: \ $4=(1+2t)^{12}$. 
\end{example}

\begin{example}
As a more complicated example, consider again the finite field $\F_{13^3} = \F_{13}[t]/(t^3 + 2t + 11)$.

Making use of the extended Euler's criterion we investigate the existence of  $r^{\rm th}$ powers in $\F_{13^3}^*$ for $r=12, 61, 1098$.

The computation
$$
(5+7t)^{(13^3-1)/12}=1,
$$
shows that $5+7t$ is a power of degree 12. In fact: $(t^2)^{12}=5+7t.$

The computation
$$
(5+3t+7t^2)^{(13^3-1)/61} = (5+3t+7t^2)^{36} = 1,
$$
shows that  $5+3t+7t^2$ is a power of degree 61. In fact: 
$$
(6+2t)^{61}=5+3t+7t^2.
$$

The computation
$$
12^{(13^3-1)/1098} = 12^{2} = 1,
$$
shows that  $12$ is a power of degree 1098. In fact: $t^{1098}=12$. 
\end{example}

\bibliographystyle{plain}

\begin{thebibliography}{1}
\bibitem{davenport} H. {Davenport}, {The Higher Arithmetic. An Introduction to the Theory of Numbers}, {Dover Publications, Inc., New York}, {1983}.
\bibitem{IrelandRosen} K. {Ireland} and M. {Rosen}, {A Classical Introduction to Modern Number Theory}, {Springer}, {1982}.
\bibitem{lang} S. {Lang}, {Algebra}, {Springer}, {2005}.
\end{thebibliography}

\end{document}